\documentclass[psamsfonts]{amsart}

\usepackage{amssymb,amsfonts}
\usepackage[all,arc]{xy}
\usepackage{enumerate}
\usepackage{mathrsfs}
\usepackage{graphicx}

\usepackage[mathscr]{euscript}

\usepackage{orcidlink}

\newtheorem{theorem}{Theorem}[section]
\newtheorem{corollary}[theorem]{Corollary}
\newtheorem{proposition}[theorem]{Proposition}
\newtheorem{lemma}[theorem]{Lemma}

\theoremstyle{definition}
\newtheorem{definition}[theorem]{Definition}

\newtheorem{example}[theorem]{Example}
\newtheorem{examples}[theorem]{Examples}

\newtheorem{remark}[theorem]{Remark}

\DeclareMathOperator{\ord}{ord}
\DeclareMathOperator{\sml}{sml}

\numberwithin{equation}{section}

\begin{document}
	
\title{Algebraic properties of overflow semirings}
	
\author{Peyman Nasehpour\orcidlink{0000-0001-6625-364X}}
	
\address{Peyman Nasehpour\\
		Education Department\\
		The New York Academy of Sciences\\
		New York, NY, USA}
\email{nasehpour@gmail.com}
	
	

\begin{abstract}
We introduce the overflow semiring $S = A \uplus_{\ord} L$, extending a positive information algebra $A$ by a join-semilattice $L$, where elements of $L$ dominate $A$ and arithmetic in $L$ reduces to the join. This models overflow in computational systems and generalizes the transition from finite to infinite cardinal arithmetic. We characterize the idempotent elements of $S$ and $S[X]$, fully classify idempotent power series over cardinal numbers, describe the structure of prime and maximal ideals, compute the Krull dimension of $S$ ($\dim S = \dim A + |L|$ for well-ordered finite $L$), and establish Noetherian and Artinian criteria.
\end{abstract}

\maketitle
	
\section{Introduction}
	
Inspired by the arithmetic of the semiring of cardinal numbers, we introduce a novel extension of a positive information algebra by means of a join-semilattice. More precisely, given a positive information algebra $A$ and a join-semilattice $L$, we construct a new positive information algebra
\[
S = A \uplus_{\ord} L
\]
containing $A$ as a subsemiring and $L$ as an upper semilattice component. Conceptually, this construction abstracts the passage from natural number arithmetic to cardinal arithmetic: within the lower component $A$, operations follow the original algebraic structure, whereas once computations enter the upper component $L$, the behavior becomes order-dominated, resembling the arithmetic of infinite cardinals.

Beyond its intrinsic algebraic interest, this construction also admits a natural interpretation in engineering and computational systems. In many practical settings, ordinary arithmetic governs system behavior only within a safe operational range, while overflow push the system into a different computational regime. The semilattice component $L$ models such exceptional states, where arithmetic no longer behaves classically but is instead governed by ordered propagation rules (see Example \ref{fusionrestrictedmaxplusalgebra}).
	
Since the terminology of semiring theory is not completely standardized \cite[p. 3]{Glazek2002}, we begin by fixing the notation and terminology used throughout the paper. We also briefly summarize the main results obtained herein.

A bimagma $(H,+,\cdot)$ is called a hemiring if $(H,+,0)$ is a commutative monoid, $(H,\cdot)$ is a semigroup, multiplication distributes over addition from both sides, and $0$ is multiplicatively absorbing; that is, $h0 = 0h = 0$ for all $h \in H$. A hemiring $H$ is called commutative if $ab = ba$ for all $a,b \in H$. A hemiring $S$ is called a semiring if it possesses a nonzero multiplicative identity element $1$ \cite[p. 1]{Golan1999}. Throughout this paper, all semirings are assumed to be commutative.

A semiring $S$ is called zerosumfree if $a+b=0$ implies $a=b=0$ for all $a,b \in S$, and entire if $ab=0$ implies $a=0$ or $b=0$ for all $a,b \in S$. A semiring that is both zerosumfree and entire is called an information algebra \cite[p. 4]{Golan1999}. We say that a semiring $S$ is conical if $ab=1$ implies $a=1$ (and hence $b=1$; see Definition \ref{conicalsemiringdef}). Finally, a semiring $S$ is called positive if it is equipped with a partial order $\leq$ compatible with both addition and multiplication, and such that $0$ is the least element \cite[p. 27]{Golan2003}.
	
Let $(A, +, \cdot, \leq_A)$ be a positive information algebra and let $(L, \leq_L)$ be a join-semilattice, with $A$ and $L$ disjoint. Set $S = A \cup L$. We extend the orders on $A$ and $L$ to an order on $S$ by declaring
\[
a \leq \ell
\quad\text{for all } a \in A,\ \ell \in L,
\]
while $\ell \leq a$ never holds for $\ell \in L$ and $a \in A$. Addition and multiplication on $S$ are defined for $x,y \in S$ by
\[
x + y =
\begin{cases}
	x +_A y & \text{if } x,y \in A,\\
	\sup\{x,y\} & \text{otherwise},
\end{cases}
\qquad
x \cdot y =
\begin{cases}
	x \cdot_A y & \text{if } x,y \in A\setminus\{0\},\\
	0 & \text{if } x=0 \text{ or } y=0,\\
	\sup\{x,y\} & \text{otherwise}.
\end{cases}
\] The upper semilattice component $L$ may be interpreted as an overflow regime in which the ordinary arithmetic of the base semiring $A$ is replaced by order-dominated behavior (see Definition \ref{fusionsemiringlatticedef}). In \S\ref{sec:overflowsemiring}, we prove that the resulting structure
\[
S = A \cup L
\]
forms a positive information algebra (Theorem \ref{fusionsemiringlatticeordered}). Motivated by its interpretation as a transition from ordinary arithmetic to an overflow regime, we call
\[
S = A \uplus_{\ord} L
\]
the overflow semiring associated to $A$ and $L$ (see Definition \ref{overflowsemiring}). We also show that for every $s \in L$, the lower set $S^{\leq s}$ is a subsemiring of $S$; if $L$ is totally ordered, then so is $S^{<s}$ (cf. Proposition \ref{subsemiringsfusionsemiringlattice}), yielding a directed filtration $\{S^{\leq s}\}_{s \in L}$ of $S$ (cf. Corollary \ref{DirectedFiltration}).
	
After introducing overflow semirings, in \S\ref{sec:distinguishedelementsoverflowsemirings} we study their distinguished elements. For instance, Proposition \ref{fusionsemiringlatticeidempotent} and Proposition \ref{fusionsemiringlatticevnr} characterize idempotent and von Neumann regular elements of $S$, respectively. We then determine families of multiplicatively idempotent polynomials and formal power series over overflow semirings (see Propositions \ref{idempotentpolynomialsoverflow}, \ref{fusionsemiringlatticepowerseriesidempotent1}, and \ref{fusionsemiringlatticepowerseriesidempotent2}). Special attention is given to formal power series over the semiring of cardinal numbers, for which Corollary \ref{idempotentsformalpowerseriescardinalnumbers} provides a complete characterization of the idempotent elements. Recall that an element $a$ of a semiring $S$ is called small if $b \in S \setminus U(S)$ implies $a+b \in S \setminus U(S)$ \cite[p. 77]{Golan1999}. We denote the set of all small elements of $S$ by $\sml(S)$. In Theorem \ref{fusionsemiringlatticesmallelements}, we show that if $S=A\uplus_{\ord}L$ is the overflow semiring, then \[\sml(S) = \sml(A)\cup L.\]
	
A nonempty subset $I$ of $S$ is called an ideal, denoted $I \trianglelefteq S$, if $a+b \in I$ and $sa \in I$ for all $a,b \in I$ and $s \in S$ \cite[p. 65]{Golan1999}. The set of all ideals of a semiring $S$ is denoted by $\operatorname{Id}(S)$. An ideal is finitely generated if it is generated by a finite subset of $S$, and principal if it is generated by a single element $a \in S$, denoted $(a)$ \cite[p. 68]{Golan1999}. Since $S$ is commutative,
\[
(a)=aS=\{as:s\in S\}.
\] An ideal $I$ is proper if $I \neq S$, and subtractive if $x+y \in I$ and $x \in I$ imply $y \in I$ for all $x,y \in S$ \cite[p. 66]{Golan1999}. A proper ideal $P$ is called prime if $ab \in P$ implies $a \in P$ or $b \in P$ (see Corollary 7.6 in \cite{Golan1999}).
	
The final section investigates the ideal theory of overflow semirings. We first show that $S$ is always austere, meaning its only subtractive ideals are trivial (Proposition \ref{fusionsemiringlatticeaustere}), and examine extension and contraction along the inclusion $A \hookrightarrow S$ (Proposition \ref{fusionsemiringlatticeextcontractionideals}). After characterizing the ideals, primes, and maximals of $S$ (Theorems \ref{fusionsemiringlatticemaximalideal} and \ref{fusionsemiringlatticeidealstructure}), we establish the Krull dimension formula:
\[ \dim S = \dim A + \lvert L \rvert \]
(see Theorem \ref{Krulldimensionoffusion}). Finally, assuming $L$ is well-ordered, we prove that $S$ is Noetherian if and only if $A$ is Noetherian (Theorem \ref{fusionsemiringlatticeNoetherian}), and Artinian if and only if $A$ is Artinian and $L$ is a finite chain (Theorem \ref{fusionsemiringlatticeArtinian}).
	
\section{Definition of the overflow semiring}\label{sec:overflowsemiring}
	
\begin{definition}\label{fusionsemiringlatticedef}
	Let $(A, +, \cdot, \leq_A)$ be a positive information algebra and let $(L, \leq_L)$ be a join-semilattice, with $A$ and $L$ disjoint. Set $S = A \cup L$. Define an order on $S$ by:
	\begin{itemize}
		\item preserving the original orders on $A$ and $L$,
		\item $a \leq \ell$ for all $a \in A$, $\ell \in L$,
		\item $\ell \leq a$ never holds for $\ell \in L$, $a \in A$.
	\end{itemize}
	Define addition and multiplication on $S$ for $x, y \in S$ by
	\[
	x + y = \begin{cases}
		x +_A y & \text{if } x, y \in A, \\
		\sup\{x, y\} & \text{otherwise},
	\end{cases}
	\qquad
	x \cdot y = \begin{cases}
		x \cdot_A y & \text{if } x, y \in A\setminus\{0\}, \\
		0 & \text{if } x = 0 \text{ or } y = 0, \\
		\sup\{x, y\} & \text{otherwise}.
	\end{cases}
	\]
\end{definition}

\begin{lemma}\label{fusionsemiringlatticeadditionmonoid} The structure $(S,+,0)$ in Definition \ref{fusionsemiringlatticedef} is a commutative monoid.
\end{lemma}

\begin{proof}
Commutativity of $+$ is clear. For associativity \[(x+y)+z = x+(y+z),\] we consider the location of the arguments. If all three are in $A$, it holds in the semiring. If all three are in $L$, it holds in the semilattice. If one, say, $x$ is in $A$, and others are in $L$, then the value of both sides is $\sup\{y,z\}$. And if two of them are in $A$, say, $x$ and $y$, and $z$ is in $L$, then the value of both sides is $z$ (by commutativity, other permutations follow). The element $0 \in A$ is the identity: for $x \in A$, $x+0 = x$; and for $x \in L$, $x+0 = \sup\{x,0\} = x$ because $0 < x$. Hence $(S,+,0)$ is a commutative monoid.
\end{proof}

\begin{lemma}\label{fusionsemiringlatticemultiplicationmonoid} The structure $(S,\cdot,1)$ in Definition \ref{fusionsemiringlatticedef} is a commutative monoid and $0$ is its absorbing element.
\end{lemma}

\begin{proof}
First, $0$ is absorbing: for any $x \in S$, if $x \in A$ then $0 \cdot x = 0$; and if $x \in L$ then $0 \cdot x = 0$ by definition (the second case of multiplication). Hence $0 \cdot x = 0 = x \cdot 0$ for all $x \in S$ by commutativity of $\cdot$ in $S$. For associativity \[(x \cdot y) \cdot z = x \cdot (y \cdot z),\] we consider the location of the arguments. If any of $x,y,z$ is $0$, both sides are $0$ by the absorbing property of $0$ just proved. If all three are in $A$, it holds in the semiring. If all three are in $L$, it holds in the semilattice. If one is in $A$ (non-zero) and the others in $L$, both sides equal $\sup\{y,z\}$. If two are in $A$ (non-zero) and one in $L$, both sides equal the element in $L$. The element $1 \in A$ is the identity: for $x \in A$, $x \cdot 1 = x$; and for $x \in L$, $x \cdot 1 = \sup\{x,1\} = x$ because $1 < x$. Hence $(S,\cdot,1)$ is a commutative monoid and $0$ is its absorbing element.
\end{proof}

\begin{proposition}\label{fusionsemiringlatticeorderedinformationalgebra}
	The structure $(S,+,\cdot)$ in Definition \ref{fusionsemiringlatticedef} is a semiring.
\end{proposition}

\begin{proof}
	By Lemma \ref{fusionsemiringlatticeadditionmonoid}, $(S,+,0)$ is a commutative monoid. By Lemma \ref{fusionsemiringlatticemultiplicationmonoid}, $(S,\cdot,1)$ is a commutative monoid with $0$ absorbing. We only need to verify distributivity: \[x \cdot (y + z) = (x \cdot y) + (x \cdot z), \qquad\forall\, x,y,z \in S.\] The verification proceeds by cases on $x$:
	
	\emph{Case $x = 0$:} Then $x \cdot (y+z) = 0$ and $(x \cdot y)+(x \cdot z) = 0+0 = 0$.
	
	\emph{Case $x \in A \setminus \{0\}$:} 
	\begin{itemize}
		\item If $y,z = 0$: both sides are $0$.
		\item If $y = 0$, $z \in A \setminus \{0\}$: LHS $= xz$, RHS $= 0 + xz = xz$.
		\item If $y = 0$, $z \in L$: LHS $= x \cdot z = z$, RHS $= 0 + z = z$.
		\item If $y \in A \setminus \{0\}$, $z = 0$: symmetric.
		\item If $y,z \in A \setminus \{0\}$: This case holds by distributivity in $A$.
		\item If $y \in A \setminus \{0\}$, $z \in L$: LHS $= x \cdot (y + z) = x \cdot z = z$, RHS $= xy + xz = xy + z = z$ (since $z \in L$ dominates).
		\item If $y \in L$, $z = 0$: symmetric.
		\item If $y \in L$, $z \in A \setminus \{0\}$: symmetric.
		\item If $y,z \in L$: LHS $= x \cdot (y \vee z) = y \vee z$, RHS $= y + z = y \vee z$.
	\end{itemize}
	
	\emph{Case $x \in L$:}
	\begin{itemize}
		\item If $y,z = 0$: both sides are $0$.
		\item If $y = 0$, $z \in A \setminus \{0\}$: LHS $= x \cdot z = x$, RHS $= x \cdot 0 + x \cdot z = 0 + x = x$.
		\item If $y = 0$, $z \in L$: LHS $= x \cdot z = x \vee z$, RHS $= 0 + (x \vee z) = x \vee z$.
		\item If $y \in A \setminus \{0\}$, $z = 0$: symmetric.
		\item If $y,z \in A \setminus \{0\}$: by zerosumfree property for $A$, $y+z \in A \setminus \{0\}$, so LHS $= x \cdot (y+z) = x$, RHS $= x + x = x$.
		\item If $y \in A \setminus \{0\}$, $z \in L$: LHS $= x \cdot z = x \vee z$, RHS $= x + (x \vee z) = x \vee z$.
		\item If $y \in L$, $z = 0$: symmetric.
		\item If $y \in L$, $z \in A \setminus \{0\}$: symmetric.
		\item If $y,z \in L$: LHS $= x \cdot (y \vee z) = x \vee y \vee z$, RHS $= (x \vee y) + (x \vee z) = x \vee y \vee z$.
	\end{itemize} Hence distributivity holds. Thus $(S,+,\cdot)$ is a semiring, completing the proof.
\end{proof}

\begin{remark}
In Definition~\ref{fusionsemiringlatticedef}, requiring $A$ to be an information algebra (that is, zerosumfree and entire) is essential in order for $S$ to satisfy the semiring axioms. If $A$ is not zerosumfree, choose $x \in L$ and $y, z \in A \setminus \{0\}$ with $y+z=0$. Then
	\[
	x \cdot (y+z) = x \cdot 0 = 0 \neq x = (x \cdot y)+(x \cdot z),
	\]
	violating distributivity. If $A$ is not entire, choose $x \in L$ and $y, z \in A \setminus \{0\}$ with $y \cdot z = 0$. Then
	\[
	(x \cdot y) \cdot z = x \cdot z = x \neq 0 = x \cdot 0 = x \cdot (y \cdot z),
	\]
	violating associativity of multiplication. Thus both zerosumfreeness and entireness of $A$ are necessary for $S$ to be a semiring.
\end{remark}

\begin{lemma}\label{fusionsemiringlatticeorder}
	The extended relation $\leq$ on $S = A \cup L$ in Definition \ref{fusionsemiringlatticedef} is a partial order and the least element of $S$ is $0$.
\end{lemma}

\begin{proof}
We verify the three axioms:
	
	\begin{itemize}
		\item The relation $\leq$ on $S$ is reflexive because $\leq_A$ and $\leq_L$ are both reflexive.
		
		\item For antisymmetry condition of $\leq$, suppose $x \leq y$ and $y \leq x$ for $x,y \in S$.
		\begin{itemize}
			\item If $x,y \in A$ or $x,y \in L$, then antisymmetry follows from the antisymmetry condition of $\leq_A$ or $\leq_L$.
			\item If $x \in A$ and $y \in L$, then $x \leq y$ holds by definition, but $y \leq x$ never holds. Hence this case cannot occur. Similarly, $x \in L$ and $y \in A$ cannot occur with both inequalities.
		\end{itemize}
		Thus $x = y$ in all possible cases.
		
		\item For transitivity, suppose $x \leq y$ and $y \leq z$ for $x,y,z \in S$.
		\begin{itemize}
			\item If all three are in $A$ or in $L$, transitivity follows from the transitivity of $\leq_A$ or $\leq_L$.
			\item If $x \in A$ and $z \in L$, then $x \leq z$ no matter $y \in A$ or $y \in L$ because every element in $A$ is smaller than every element in $L$.
			\item The other cases never happen because any element in $L$ cannot be smaller than or equal to any element in $A$.
		\end{itemize}
		Hence $x \leq z$ in all cases.
		\end{itemize} Therefore $\leq$ is a partial order on $S$. It is clear that $0 \le x$ for all $x \in S$.
\end{proof}

\begin{theorem}\label{fusionsemiringlatticeordered}
The structure $(S,+,\cdot,\leq)$ in Definition \ref{fusionsemiringlatticedef} is a positive information algebra.
\end{theorem}

\begin{proof}
We first verify the two compatibility conditions:
	
\begin{enumerate}
		\item For any $x,y,z \in S$ with $x \leq y$, we show $x+z \leq y+z$.
		\begin{itemize}
			\item If $x,y \in A$ and $z \in A$: follows from $A$ being an ordered semiring.
			\item If $x,y \in A$ and $z \in L$: then $x+z = z$ and $y+z = z$, so equality holds.
			\item If $x \in A$, $y \in L$ (so $x \leq y$ automatically), and $z \in A$: then $x+z \in A$, $y+z = y$ (since $y \in L$ dominates), and $x+z \leq y$ because every element in $A$ is below every element in $L$.
			\item If $x \in A$, $y \in L$, and $z \in L$: then $x+z = z$, $y+z = y \vee z$, and $z \leq y \vee z$.
			\item If $x,y \in L$: then $x \leq y$ implies $x \vee z \leq y \vee z$ for any $z \in L$; if $z \in A$, then $x+z = x$, $y+z = y$, and $x \leq y$.
		\end{itemize}
		
		\item For any $x,y,z \in S$ with $x \leq y$ and $z \geq 0$, we show $x \cdot z \leq y \cdot z$.
		\begin{itemize}
			\item If $z = 0$: both sides are $0$, so equality holds.
			\item If $z \in A \setminus \{0\}$: 
			\begin{itemize}
				\item If $x,y \in A$: then $x \cdot z \leq y \cdot z$ because $A$ is an ordered semiring.
				\item If $x,y \in L$: then $xz = x$ and $yz = y$, and $x \leq y$ implies $x \cdot z \leq y \cdot z$.
				\item If $x \in A$ and $y \in L$: then $x \leq y$ automatically. Here $xz \in A$ (since $x,z \in A$) and $yz = y \in L$. Since every element of $A$ is below every element of $L$, we have $xz \leq y$, i.e., $x \cdot z \leq y \cdot z$.
				\item If $x \in L$ and $y \in A$: this cannot happen because $x \leq y$ would imply an $L$-element is below an $A$-element, contradicting $A < L$.
			\end{itemize}
			\item If $z \in L$: then for any $w \in A\setminus\{0\}$, we have $w \cdot z = z$; for $w \in L$, we have $w \cdot z = w \vee z$; and $0 \cdot z = 0$. Thus:
			\begin{itemize}
				\item If $x = 0$, then $x \cdot z = 0$. Since $x \leq y$, we have $0 \leq y$, so $y$ could be $0$ or nonzero. If $y = 0$, then $y \cdot z = 0$ and $0 \leq 0$ holds. If $y \neq 0$, then either $y \in A$, and so $y \cdot z = z$ and $0 \leq z$ holds because $z \in L$; or $y \in L$, and so $y \cdot z = y \vee z$ which is obviously $\geq 0$. 
				\item Let $x \neq 0$. Since $x \leq y$ and $x \neq 0$, we must have $y \neq 0$ (otherwise $x \leq 0$ would force $x = 0$, a contradiction). 
				\begin{itemize}
					\item If $x,y \in A$, then $xz = z$ and $yz = z$, so equality holds.
					\item If $x \in A$ and $y \in L$, then $xz = z$ while $yz = y \vee z$, and the inequality $z \leq y \vee z$ holds clearly.
					\item The case $x \in L$ and $y \in A$ cannot happen because $x \leq y$ would imply an element in $L$ is below an element in $A$, a contradiction: The assumption $x \leq y$ for $x \in L$ and $y \in A$ would imply $x = y$, contradicting the disjointness of $A$ and $L$.
					\item If $x,y \in L$, then $xz = x \vee z$ and $yz = y \vee z$, and $x \leq y$ implies $x \vee z \leq y \vee z$ in the semilattice.
				\end{itemize}
			\end{itemize}
			Therefore in all cases, $x \cdot z \leq y \cdot z$.
		\end{itemize}
	\end{enumerate}
	
	Thus, $(S,+,\cdot,\leq)$ is a positive semiring. We then show $S$ is zerosumfree by showing that if $0 < x$, then $0 < x+y$ for any $y \geq 0$. Consider cases:
	\begin{itemize}
		\item $x,y \in A$: then $x+y \in A$ and $x+y > 0$ because $A$ is zerosumfree.
		\item $x \in A$, $y \in L$: then $x+y = y \in L$, and since every element of $A$ is smaller than every element of $L$, we have $y > 0$, so $x+y > 0$.
		\item $x \in L$, $y \in A$: then $x+y = x \in L$, and $x > 0$ by assumption.
		\item $x,y \in L$: then $x+y = x \vee y \in L$, hence $x+y > 0$.
	\end{itemize}
	Thus $0 < x+y$ in all cases. Similarly one may prove that $S$ is entire by showing that the conditions $0 < x$ and $0 < y$ imply $0 < xy$. Hence $S$ is a positive information algebra.
\end{proof}

\begin{definition}\label{overflowsemiring}
	We call the positive information algebra $(S,+,\cdot,\leq)$ obtained by extending a positive information algebra $A$ with a join-semilattice $L$ (see Theorem \ref{fusionsemiringlatticeordered}) the overflow semiring, denoted by $A \uplus_{\ord} L$.
\end{definition}

\begin{examples}\label{fusionrestrictedmaxplusalgebra}
To illustrate the construction in Definition~\ref{overflowsemiring}, we present the following examples.
	
\begin{enumerate}
		
\item In Zermelo--Fraenkel set theory with the Axiom of Choice, consider the positive information algebra $\mathbb{N}_0$ together with the join-semilattice $\mathcal{L}$ of all infinite cardinal numbers. The overflow semiring $\mathscr{C} = \mathbb{N}_0 \uplus_{\ord} \mathcal{L}$, equipped with cardinal addition and multiplication, forms a class-semiring satisfying the axioms of a positive information algebra and is totally ordered by the usual ordering of cardinals (see Chapter~4 of \cite{Monk1969}).
		
\item It is straightforward to verify that the restricted max-plus algebra $A=[-\infty,0]$ (with addition $\oplus=\max$ and multiplication $\odot=+$) is a positive information algebra. Let $L=\mathbb{N}$ denote the set of positive integers equipped with the usual order. In the overflow semiring
$S=A\uplus_{\ord}L$, addition is simply given everywhere by $\max$. Multiplication is defined by
		\[
		x\odot y=
		\begin{cases}
			x+y, & \text{if } x,y\in A\setminus\{-\infty\},\\[4pt]
			-\infty, & \text{if } x=-\infty \text{ or } y=-\infty,\\[4pt]
			\max\{x,y\}, & \text{otherwise}.
		\end{cases}
		\] The resulting algebraic structure exhibits a surprisingly hybrid behavior between tropical arithmetic and ordered overflow dynamics.
		
\item The set $A=[0,1]$, equipped with $\max$ as addition, $\min$ as multiplication, and the usual order on the real numbers, forms a positive information algebra. Extending this structure by the singleton join-semilattice $L=\{\ell\}$ yields the overflow semiring $S=A\uplus_{\ord}\{\ell\}$, whose operations are given by
		\[
		x\oplus y=
		\begin{cases}
			\max(x,y), & x,y\in[0,1],\\
			\ell, & \text{otherwise},
		\end{cases}
		\qquad
		x\odot y=
		\begin{cases}
			\min(x,y), & x,y\in[0,1]\setminus\{0\},\\
			0, & x=0 \text{ or } y=0,\\
			\ell, & \text{otherwise}.
		\end{cases}
		\] This construction admits a natural interpretation in resource allocation:
		
\begin{itemize}
\item $[0,1]$ represents normalized resource levels ($0$ meaning no resources and $1$ full capacity);
\item $\ell$ represents an overflow state beyond system capacity;
\item $\max$ models competitive aggregation by selecting the larger demand;
\item $\min$ models cooperative interaction through bottleneck behavior;
\item once the system becomes overflowed, it remains overflowed unless reset by multiplication with $0$.
\end{itemize} Thus, $S$ models a resource system in which exceeding capacity irreversibly triggers a failure state, while multiplication by $0$ acts as a complete reset mechanism.

\item The Viterbi semiring $([0,1], \max, \cdot, \leq)$ (see p. 3 in \cite{Golan2003}) is a positive information algebra widely used in dynamic programming. For a join-semilattice $(L,\leq_L)$, the overflow semiring 
\[
V = [0,1] \uplus_{\text{ord}} L
\]
is also a positive information algebra. It is additively idempotent since addition is the $\sup$ operation. This construction extends standard Viterbi optimization to scenarios with multi-tiered or infinite utility states by gluing $L$ above $[0,1]$ to represent degrees of non-standard certainty. For instance, if $L = \{+\infty\}$, then 
\[
V = [0,1] \cup \{+\infty\}
\]
where $+\infty$ acts as a ``Super-Pass'' token in path-finding. If a path possesses this token, no amount of scaling by fractional probabilities can degrade it, unless it hits a dead end ($0$).
		
\item The Boolean semiring $(\mathbb{B}=\{0,1\},+,\cdot, \leq)$ forms a positive information algebra. Extending this structure by a singleton join-semilattice $\{m\}$ yields the overflow semiring $S=\mathbb{B}\uplus_{\ord}\{m\}$. Structurally, the additive behavior of $S$ closely resembles Bochvar's internal three-valued logic, which introduces a third truth value $m$ (meaningless) governed by a contagion principle satisfying $\neg m=m$, and for which an implication $x\to y$ evaluates to $m$ whenever either operand equals $m$ \cite[p. 80]{Bergmann2008}. Interpreting semiring addition through the classical equivalence $a+b \equiv \neg a \to b$, one obtains an additive Cayley table isomorphic to the additive structure of $S$. A fundamental structural divergence nevertheless appears in the multiplicative behavior. Bochvar's logic enforces complete logical contagion, requiring $0\cdot m=m$, whereas the overflow semiring preserves the algebraic annihilator law $0\cdot m=0$. From a computational perspective, this models short-circuit evaluation and local fault isolation: once a controlling gate is closed ($0$), a downstream exceptional state ($m$) is never propagated or executed. By privileging algebraic annihilation over universal logical contagion, the semiring $S$ ensures that a masked or inactive computational channel cannot be corrupted by a latent runtime fault.
\end{enumerate}
\end{examples}

\begin{proposition}\label{subsemiringsfusionsemiringlattice}
	Let $S = A \uplus_{\operatorname{ord}} L$ be the overflow semiring. For any $s \in L$:
	\begin{enumerate}
		\item The set $S^{\leq s} = \{x \in S : x \leq s\}$ is a subsemiring of $S$.
		\item If $L$ is totally ordered, then $S^{< s} = \{x \in S : x < s\}$ is also a subsemiring of $S$.
	\end{enumerate}
\end{proposition}

\begin{proof}
	(1) $0,1 \in S^{\leq s}$ because $A < L$. If $x,y \in S^{\leq s}$, then either $x,y \in A$ and $x+y, xy \in A < s$, or at least one lies in $L$, in which case both operations reduce to $\sup\{x,y\} \leq s$. Hence closure holds.
	
	(2) Under total order, if $x,y < s$ then $\sup\{x,y\} < s$; the same case analysis as above gives closure for addition and multiplication, and $0,1 < s$ because $A < L$. Thus $S^{<s}$ is a subsemiring.
\end{proof}

\begin{example}\label{necessity-total-order-strict}
	The total order assumption in Proposition \ref{subsemiringsfusionsemiringlattice}(2) is essential. Let $A = \mathbb{N}_0$ and let $L = \{\ell_1, \ell_2, s\}$ be the join-semilattice with $\ell_1 < s$, $\ell_2 < s$, and $\ell_1 \vee \ell_2 = s$ (so $L$ is not totally ordered). In $S = A \uplus_{\operatorname{ord}} L$, we have $\ell_1, \ell_2 \in S^{<s}$, but
	\[
	\ell_1 + \ell_2 = \sup\{\ell_1, \ell_2\} = \ell_1 \vee \ell_2 = s \notin S^{<s}.
	\]
	Thus $S^{<s}$ is not closed under addition and hence is not a subsemiring.
\end{example}

\begin{corollary}[Directed Filtration by $L$]\label{DirectedFiltration}
Let $S = A \uplus_{\ord} L$ be the overflow semiring. The family of subsemirings $\{S^{\leq s}\}_{s \in L}$ forms an increasing directed filtration of $S$, such that
	\[
	S = \bigcup_{s \in L} S^{\leq s}.
	\]
\end{corollary}

\begin{proof}
Let $s, t \in L$ with $s \leq t$. If $x \in S^{\leq s}$, then $x \leq s \leq t$ implies $x \in S^{\leq t}$. Thus, $S^{\leq s} \subseteq S^{\leq t}$, establishing that the family is tracking the poset order. Now, let $x \in S$. If $x \in A$, then $x \leq s$ for all $s \in L$, placing $x \in \bigcup_{s \in L} S^{\leq s}$. If $x \in L$, then by reflexivity $x \leq x$, placing $x \in S^{\leq x}$. Hence, the identity $S = \bigcup_{s \in L} S^{\leq s}$ holds, completing the proof.
\end{proof}

\section{Distinguished elements of the overflow semirings}\label{sec:distinguishedelementsoverflowsemirings}

\begin{proposition}\label{fusionsemiringlatticeidempotent}
Let $I^+(S)$ and $I^\times(S)$ denote the additively and multiplicatively idempotent elements of $S$, respectively. Then for the overflow semiring $S = A \uplus_{\ord} L$,
	\[
	I^+(S) = I^+(A) \cup L \qquad \text{and} \qquad I^\times(S) = I^\times(A) \cup L.
	\]
\end{proposition}

\begin{proof}
If $x \in L$, then $x+x = \sup\{x,x\} = x$ and $x \cdot x = \sup\{x,x\} = x$, so $L \subseteq I^+(S), \, I^\times(S)$. If $x \in A$, additive or multiplicative idempotency of $x$ in $S$ follows from that in $A$. Conversely, no element outside $I^+(A) \cup L$ can be additively idempotent, and similarly for multiplication. Hence the equalities hold.
\end{proof}

An element $a$ of a semiring $S$ is called multiplicatively invertible (or a unit) if $ab = 1$ for some $b \in S$. The group of all units of $S$ is denoted by $U(S)$.

\begin{proposition}\label{fusionsemiringlatticeunits}
For the overflow semiring $S = A \uplus_{\operatorname{ord}} L$, the group of units satisfies $U(S) = U(A)$.
\end{proposition}

\begin{proof}
If $\ell \in L$, then $\ell y \in L \cup \{0\}$ for every $y \in S$, so $\ell y \neq 1$. Hence no element of $L$ is a unit, and therefore $U(S)\subseteq A$. Now multiplication on $A$ is unchanged in $S$. Thus every $a\in U(A)$ remains invertible in $S$, so $U(A)\subseteq U(S)$. Conversely, if $a\in A$ is invertible in $S$, its inverse must lie in $A$ since elements of $L$ are not units. Hence $a\in U(A)$. Hence, $U(S)=U(A)$ as claimed.
\end{proof}

Recall that a multiplicative monoid $M$ is called conical if $ab = 1$ implies $a = b = 1$ for all $a,b \in M$ \cite[p. xvii]{Cohn2006}. We extend this notion to semirings as follows.

\begin{definition}\label{conicalsemiringdef}
We call	a semiring $S$ conical if $ab = 1$ implies $a = b = 1$ for all $a,b \in S$.
\end{definition}

\begin{proposition}
	If $A$ is conical, then so is the overflow semiring $S = A \uplus_{\ord} L$ and the polynomial semiring $S[X]$.
\end{proposition}

\begin{proof}
	Suppose $x \cdot y = 1$ in $S$. Since $1 \in A$ and any multiplication involving an element of $L$ lies in $L$ (or is $0$ if the other factor is $0$), we must have $x, y \in A$ and $x \cdot y = 1$. Conicality of $A$ then yields $x = y = 1$. Now let $f, g \in S[X]$ with $fg = 1$. Since $S[X]$ is entire, \[\deg(f) + \deg(g) = \deg(fg) = \deg(1) = 0,\] so $\deg(f) = \deg(g) = 0$. Thus $f$ and $g$ are constant polynomials, i.e., $f = a$, $g = b$ for some $a, b \in S$. Then $ab = 1$ in $S$, so by the first part $a = b = 1$. Hence $f = g = 1$, proving $S[X]$ is conical.
\end{proof}

\begin{example}\label{Cconical}
	The following are examples of conical positive information algebras:
	
	\begin{itemize}
		\item The class-semiring of cardinal numbers $\mathscr{C}$ and the semiring $S$ obtained from the restricted max-plus algebra in Examples \ref{fusionrestrictedmaxplusalgebra} are both conical. However, they are not isomorphic because $S$ is additively idempotent ($\max(x,x) = x$), whereas $\mathscr{C}$ is not (since $\kappa + \kappa = \kappa$ holds for infinite cardinals but fails for finite nonzero cardinals).
		
		\item Every bounded distributive lattice $(L,\vee,\wedge,0,1)$ is a conical positive semiring; if it is totally ordered (a chain), it is also a conical information algebra.
		
		\item The Viterbi positive information algebra $([0,1], \max, \cdot, \leq)$ illustrated in Examples \ref{fusionrestrictedmaxplusalgebra} is conical.
	\end{itemize}
	
\end{example}

\begin{definition}\label{smallelementsdef}
An element $a$ of a semiring $S$ is called small if \[b \in S \setminus U(S) \implies a+b \in S \setminus U(S).\] The set of all small elements of $S$ is denoted by $\sml(S)$.
\end{definition}

\begin{theorem}\label{fusionsemiringlatticesmallelements}
Let $S=A\uplus_{\ord}L$ be the overflow semiring. Then \[\sml(S) = \sml(A)\cup L.\]
\end{theorem}

\begin{proof}
First let $\ell\in L$ and let $b \notin U(S)$. Since $L$ is absorbing under the overflow addition, $\ell+b\in L$, and hence $\ell+b\notin U(S)$. Thus $\ell\in \sml(S)$, so $L\subseteq \sml(S)$. Now let $a\in \sml(A)$ and let $b\notin U(S)$. Since $U(S)=U(A)$ (Proposition \ref{fusionsemiringlatticeunits}), either $b\in A \setminus U(A)$ or $b\in L$.

\begin{itemize}
	\item If $b\in A \setminus U(A)$, then $a+b\in A \setminus U(A)$ because $a\in \sml(A)$, and so, $a+b \notin U(S)$.
	
	\item If $b\in L$, then $a+b\in L$, so again $a+b\notin U(S)$.
\end{itemize} Thus $a\in \sml(S)$, proving $\sml(A)\cup L \subseteq \sml(S)$. 

Conversely, let $a\in \sml(S)$ and assume $a\notin L$. Then $a\in A$. We show that $a\in \sml(A)$. Let $b\in A\setminus U(A)$. Since $U(A)=U(S)$, we have $b\notin U(S)$. Because $a\in \sml(S)$,
\[
a+b \notin U(S).
\]
Moreover, since $a,b\in A$, the operations of $S$ restrict to those of $A$, hence the sum $a+b$ computed in $S$ lies in $A$ and coincides with the sum in $A$. Therefore,
\[
a+b \in A \setminus U(A),
\]
so $a\in \sml(A)$. Thus $\sml(S)\subseteq \sml(A)\cup L$.
\end{proof}

\begin{definition}
An element $x$ in a semiring is called additively (resp., multiplicatively) (von Neumann) regular if there exists an element $y$ such that $x + y + x = x$ (resp., $x \cdot y \cdot x = x$). We denote the sets of these additively and multiplicatively regular elements by $\operatorname{VNR}^+(S)$ and $\operatorname{VNR}^\times(S)$, respectively.
\end{definition}

\begin{proposition}\label{fusionsemiringlatticevnr}
	For the overflow semiring $S = A \uplus_{\ord} L$, the sets of additively and multiplicatively regular elements satisfy
	\[
	\operatorname{VNR}^+(S)
	=
	\operatorname{VNR}^+(A)\cup L,
	\qquad
	\operatorname{VNR}^\times(S)
	=
	\operatorname{VNR}^\times(A)\cup L.
	\]
\end{proposition}

\begin{proof}
	For every $x\in L$, taking $y=x$ gives
	\[
	x+x+x=x, \qquad x\cdot x\cdot x=x,
	\]
	since both operations on $L$ are the join. Hence
	\[
	L\subseteq \operatorname{VNR}^+(S)\cap \operatorname{VNR}^\times(S).
	\]
	If $x\in \operatorname{VNR}^+(A)$ or $x\in \operatorname{VNR}^\times(A)$, the corresponding regularity equation in $A$ remains valid in $S$. Thus
	\[
	\operatorname{VNR}^+(A)\cup L \subseteq \operatorname{VNR}^+(S),
	\qquad
	\operatorname{VNR}^\times(A)\cup L \subseteq \operatorname{VNR}^\times(S).
	\]
	Conversely, let $x\in A$. For the multiplicative case, if $x=0$ then $x\in \operatorname{VNR}^\times(A)$; otherwise, if $x\cdot y\cdot x=x$ with $y\in L$, then
	\[
	x\cdot y\cdot x = \sup\{x,y\}=y,
	\]
	forcing $y=x$, a contradiction. Hence $y\in A$, so $x\in \operatorname{VNR}^\times(A)$. Similarly, if $x+y+x=x$ with $y\in L$, then
	\[
	x+y+x=\sup\{x,y\}=y,
	\]
	again a contradiction; hence $y\in A$ and $x\in \operatorname{VNR}^+(A)$. Therefore
	\[
	\operatorname{VNR}^+(S)=\operatorname{VNR}^+(A)\cup L,
	\qquad
	\operatorname{VNR}^\times(S)=\operatorname{VNR}^\times(A)\cup L,
	\] completing the proof.
\end{proof}

\begin{proposition}\label{idempotentpolynomialsoverflow}
Let $S = A \uplus_{\ord} L$ be the overflow semiring. Then the multiplicatively idempotent elements of $S[X]$ are the constant polynomials in $I^\times(A) \cup L$.
\end{proposition}

\begin{proof}
If $f^2 = f$, then $2\deg(f) = \deg(f)$ because $S$ is entire. Hence $\deg(f) = 0$, so $f$ is a constant polynomial, say $f = c \in S$. Then $c^2 = c$, i.e., $c \in I^\times(A) \cup L$ by Proposition \ref{fusionsemiringlatticeidempotent}. Conversely, any constant polynomial $c$ with $c \in I^\times(A) \cup L$ satisfies $c^2 = c$.
\end{proof}

\begin{proposition}\label{idempotentformalpowerseriesconstantzero}
Let $S$ be any semiring. An element $f = \sum_{n=0}^{\infty} a_n X^n$ with $a_0 = 0$ is multiplicatively idempotent in the formal power series semiring $S[[X]]$ if and only if $f = 0$.
\end{proposition}

\begin{proof}
If $f = 0$, then clearly $f^2 = f$. Conversely, suppose $f^2 = f$ and $a_0 = 0$. Then for all $n \ge 0$,
	\[
	a_n = \sum_{k=0}^n a_k a_{n-k}.
	\] For $n = 0$, we have $a_0 = a_0^2 = 0$, which holds. Assume inductively that $a_1 = \cdots = a_{n-1} = 0$. Then for $n \ge 1$,
	\[
	a_n = a_0 a_n + a_n a_0 + \sum_{k=1}^{n-1} a_k a_{n-k} = 0.
	\] Thus by induction, $a_n = 0$ for all $n \ge 1$, so $f = 0$.
\end{proof}

\begin{theorem}\label{idempotentformalpowerseriesovercardinals1}
Let $(A,+,\cdot,\leq_A)$ be an additively cancellative positive information algebra and let $(L,\leq_L)$ be a chain. Let $S = A \uplus_{\mathrm{ord}} L$ be the overflow semiring. Consider a formal power series $f = \sum_{n=0}^{\infty} a_n X^n \in S[[X]]$ with $a_0 = 1$. If $f \neq 1$, then $f^2 = f$ if and only if:
	\begin{enumerate}
		\item every non-zero coefficient $a_n$ with $n \ge 1$ lies in $L$,
		\item the set $\{\, n \ge 1 \mid a_n \neq 0 \,\}$ is infinite,
		\item for every $n \ge 1$, $a_n \ge \sup_{1 \le i \le n-1} a_i a_{n-i}$.
	\end{enumerate}
\end{theorem}

\begin{proof}
Assume $f^2 = f$ and $f \neq 1$. Equating the coefficients of $X^n$ for $n \ge 1$ yields the convolution identity \[a_n = a_n + a_n + \sum_{1 \le i \le n-1} a_i a_{n-i}.\] Suppose $a_n \in A$ for some $n \ge 1$. Since $A < L$ in $S$, a sum in $S$ evaluates to an element of $A$ if and only if all its constituent summands belong to $A$. Thus, $a_i a_{n-i} \in A$ for all $i$, meaning the identity holds completely within the additively cancellative base semiring $A$. Canceling $a_n$ from both sides gives $0 = a_n + \sum_{1 \le i \le n-1} a_i a_{n-i}$. Because $A$ is zerosumfree, this forces $a_n = 0$, which established condition (1).

With condition (1) verified, any addition involving non-zero coefficients for $n \ge 1$ is governed entirely by the supremum operation defining the overflow structure of $S$. The convolution equation can thus be rewritten using the lattice join as $a_n = a_n \lor \sup_{1 \le i \le n-1} a_i a_{n-i}$, which directly implies $a_n \ge \sup_{1 \le i \le n-1} a_i a_{n-i}$, proving condition (3).
	
For condition (2), since $f \neq 1$, there exists some $m \ge 1$ such that $a_m \in L \setminus \{0\}$. By condition (3) and the multiplicative idempotency of $L$, it follows that $a_{2m} \ge a_m a_m = a_m$. By induction, $a_{km} \neq 0$ for all $k \ge 1$, which guarantees that the set of non-zero indices is infinite.
	
\medskip \noindent Conversely assume conditions (1)--(3) hold. We check the convolution step for $n \ge 1$. If $a_n = 0$, condition (3) forces $a_i a_{n-i} = 0$ for all $1 \le i \le n-1$, reducing the identity to $0 = 0$. If $a_n \neq 0$, then $a_n \in L$ by condition (1). For each $1 \le i \le n-1$, if either factor is zero, the product vanishes; if both are non-zero, they belong to $L$, yielding $a_i a_{n-i} = \sup\{a_i, a_{n-i}\} \in L$. Because $L$ is a chain, this supremum is a well-defined maximum dominated by $a_n$ via condition (3). It follows that $a_n \lor \sup_{1 \le i \le n-1} a_i a_{n-i} = a_n$. Thus, the convolution equality holds coefficientwise, meaning $f^2 = f$, while condition (2) ensures $f \neq 1$.
\end{proof}

\begin{theorem}\label{idempotentformalpowerseriesovercardinals2}
Let $S = A \uplus_{\mathrm{ord}} L$ be an overflow semiring where $L$ is a chain. Consider a formal power series $f = \sum_{n=0}^{\infty} a_n X^n \in S[[X]]$ with $a_0 = \lambda \in L$. Then $f^2 = f$ if and only if for every $n \ge 1$:
	\begin{enumerate}
		\item $a_n \in \{0\} \cup \{x \in L \mid x \ge \lambda\}$,
		\item $a_n \ge \sup_{1 \le i \le n-1} a_i a_{n-i}$ whenever $a_n \neq 0$,
		\item $\sup_{1 \le i \le n-1} a_i a_{n-i} = 0$ whenever $a_n = 0$.
	\end{enumerate}
\end{theorem}

\begin{proof}
Assume $f^2 = f$ with $a_0 = \lambda \in L$. For $n \ge 1$, the convolution yields \[a_n = \lambda a_n + a_n \lambda + \sum_{1 \le i \le n-1} a_i a_{n-i}.\] If $a_n \in A \setminus \{0\}$, then $\lambda a_n = \sup\{\lambda, a_n\} = \lambda \in L$. This introduces an element of $L$ to the right-hand side, forcing the entire sum into $L$, which contradicts $a_n \in A$. Thus, any non-zero coefficient $a_n$ must lie in $L$. For any $a_n \in L$, we have \[\lambda a_n = \sup\{\lambda, a_n\} = \lambda \lor a_n.\] Since addition involving $L$ is the supremum operation, the convolution identity reduces to $a_n = \lambda \lor a_n \lor \sup_{1 \le i \le n-1} a_i a_{n-i}$. This immediately implies $a_n \ge \lambda$ and $a_n \ge \sup_{1 \le i \le n-1} a_i a_{n-i}$, establishing conditions (1) and (2). If $a_n = 0$, then $\lambda a_n = 0$, reducing the identity to $0 = \sum_{1 \le i \le n-1} a_i a_{n-i}$. Since $S$ is zerosumfree, this requires each intermediate summand to vanish, proving condition (3).
	
\medskip\noindent Conversely, if conditions (1)--(3) hold, we evaluate the convolution for $n \ge 1$. If $a_n = 0$, condition (3) ensures the intermediate terms vanish, reducing the identity to $0 = 0$. If $a_n \neq 0$, then $a_n \in L$ and $a_n \ge \lambda$ by condition (1), which implies $\lambda a_n + a_n \lambda = a_n \lor a_n = a_n$. By condition (2), $a_n$ dominates the supremum of the intermediate products, meaning the entire right-hand side simplifies via the lattice join to $a_n \lor \sup_{1 \le i \le n-1} a_i a_{n-i} = a_n$. Thus, $f^2 = f$ holds coefficientwise.
\end{proof}

\begin{corollary}\label{idempotentsformalpowerseriescardinalnumbers}
	Let $\mathscr{C}$ be the semiring of cardinal numbers. A formal power series $f = \sum_{n=0}^{\infty} a_n X^n \in \mathscr{C}[[X]]$ is multiplicatively idempotent if and only if its constant term $a_0$ is an idempotent cardinal--namely $a_0 = 0$, $a_0 = 1$, or an infinite cardinal $\lambda$--and except the case $f=1$, its higher coefficients satisfy the conditions of Proposition~\ref{idempotentformalpowerseriesconstantzero}, Theorem~\ref{idempotentformalpowerseriesovercardinals1}, or Theorem~\ref{idempotentformalpowerseriesovercardinals2}, respectively.
\end{corollary}

\begin{proof}
	Equating the constant terms of $f^2 = f$ yields $a_0^2 = a_0$. In $\mathscr{C}$, the scalar idempotents are precisely $0$, $1$, and all infinite cardinals. Since $\mathscr{C}$ naturally partitions into the overflow structure $\mathbb{N}_0 \uplus_{\mathrm{ord}} L$, where $\mathbb{N}_0$ is an additively cancellative, zerosumfree semiring and $L$ is the chain of infinite cardinals, the classification follows by exhausting these three algebraic regimes. If $a_0 = 0$, Proposition~\ref{idempotentformalpowerseriesconstantzero} forces $f = 0$. If $a_0 = 1$, the series is governed by the conditions of Theorem~\ref{idempotentformalpowerseriesovercardinals1}. Finally, if $a_0 = \lambda \in L$, the conditions reduce to those of Theorem~\ref{idempotentformalpowerseriesovercardinals2}. Necessity and sufficiency in each case flow directly from the corresponding results.
\end{proof}

Recall that an element $a$ of a semiring $A$ is called multiplicatively subidempotent if $a+a^2 = a$. A semiring is called multiplicatively subidempotent if each of its elements is multiplicatively subidempotent \cite[p.~8]{Golan2003}.

\begin{proposition}\label{fusionsemiringlatticetimessubidempotent}
Let $A$ be a multiplicatively subidempotent positive information algebra. Then the overflow semiring $S = A \uplus_{\ord} L$ is also multiplicatively subidempotent.
\end{proposition}

\begin{proof}
If $s \in A$, then $s$ is multiplicatively subidempotent by hypothesis. If $s = \ell \in L$, then multiplication in $L$ is join, so
	\[
	\ell + \ell^2 = \sup\{\ell, \sup\{\ell,\ell\}\} = \ell.
	\]
	Hence every element of $S$ is multiplicatively subidempotent, as claimed.
\end{proof}

\begin{remark}
Golan describes that the Viterbi semiring $V$ (see Example~\ref{fusionrestrictedmaxplusalgebra}) is not only additively idempotent and multiplicatively subidempotent, but also satisfies the following condition:
	\begin{align}\label{Viterbirarecondition}
		a+b \geq ab, \qquad \forall\, a,b \in V = [0,1].
	\end{align}
Note that in the Viterbi semiring, addition is the max operation.  
\end{remark}

\begin{proposition}
Let $A$ be a positive information algebra satisfying the condition~\eqref{Viterbirarecondition}. Then the overflow semiring $S = A \uplus_{\ord} L$ satisfies the same condition.
\end{proposition}

\begin{proof}
Let $a,b \in S$. If $a,b \in A$, then $a+b = a+_A b$ and $ab = a\cdot_A b$, so $a+b \ge ab$ by hypothesis. If $a \neq 0$ and $b \in L$, then both operations give $\sup\{a,b\}$, so $a+b = ab$. If $a = 0$ and $b \in L$, then $a+b = \sup\{0,b\} = b$ and $ab = 0$. Because $0 \le b$ in $S$ (every element of $A$ is $\le$ every element of $L$), we have $b \ge 0$, i.e., $a+b \ge ab$. The cases with $a \in L$, $b \in A$ are symmetric. Thus $a+b \ge ab$ holds for all $a,b \in S$.
\end{proof}

\section{Ideal theory of the overflow semiring}\label{sec:idealtheoryoverflowsemiring}

Recall that a semiring $S$ is called austere if its only subtractive ideals are $\{0\}$ and $S$ \cite[p. 71]{Golan1999}.

\begin{proposition}\label{fusionsemiringlatticeaustere}
The overflow semiring $S = A \uplus_{\ord} L$ is austere.
\end{proposition}

\begin{proof}
Let $I$ be a nonzero proper ideal of $S$. Pick $x \in I \setminus \{0\}$. If $x \in A$, then for any $\ell \in L$,
	\[
	x\ell = \sup\{x,\ell\} = \ell \in I,
	\]
	so $L \subseteq I$. Since $1 + \ell = \ell \in I$ and $1 \notin I$, it follows that $I$ is not subtractive. If $x \in L$, then $1 + x = x \in I$ with $1 \notin I$, again $I$ is not subtractive. Hence, the only subtractive ideals are $\{0\}$ and $S$.
\end{proof}

\begin{proposition}\label{fusionsemiringlatticeextcontractionideals}
	Let $S = A \uplus_{\ord} L$ be the overflow semiring. Then:
	
	\begin{itemize}
		\item If $I \trianglelefteq A$, then $I \cup L$ is an ideal of $S$, and the map
		\[
		I \mapsto I \cup L
		\]
		is injective from $\mathrm{Id}(A)$ into $\mathrm{Id}(S)$.
		
		\item If $J \trianglelefteq S$, then $J \cap A$ is an ideal of $A$.
		
		\item The contraction map $J \mapsto J \cap A$ is injective on the set of ideals of $S$ containing $L$.
	\end{itemize}
\end{proposition}

\begin{proof}
	Let $I \trianglelefteq A$. Then $I \cup L$ is closed under addition and absorbs multiplication by elements of $S$, hence is an ideal of $S$. Injectivity of $I \mapsto I \cup L$ follows from $(I \cup L) \cap A = I$.
	
	Let $J \trianglelefteq S$. If $a,b \in J \cap A$, then $a+b \in J \cap A$, and if $a \in J \cap A$ and $x \in A$, then $ax \in J \cap A$, so $J \cap A \trianglelefteq A$.
	
	If $J_1, J_2 \trianglelefteq S$ with $L \subseteq J_1, J_2$ and $J_1 \cap A = J_2 \cap A$, then $J_i = (J_i \cap A) \cup L$, hence $J_1 = J_2$.
\end{proof}

\begin{theorem}\label{fusionsemiringlatticeidealstructure}
	Let $S = A \uplus_{\ord} L$ be the overflow semiring with $L$ well-ordered, and let $J$ be a nonzero ideal of $S$.
	\begin{enumerate}
		\item[Type I:] If $J$ contains a nonzero element of $A$, then $J = J_A \cup L$ where $J_A = J \cap A \trianglelefteq A$. Moreover, $J$ is prime if and only if $J_A$ is prime in $A$.
		
		\item[Type II:] If $J$ contains no nonzero element of $A$, then $J = \{0\} \cup \{\ell \in L : \ell \ge \ell_0\}$ where $\ell_0 = \min(J \cap L)$, and $J$ is prime.
	\end{enumerate}
\end{theorem}

\begin{proof}
	We treat the two cases separately.
	
	\begin{enumerate}
		\item Let $a \in J \cap A$ be nonzero and $\ell \in L$. Since $J$ is an ideal of $S$, 
		\[
		\ell \cdot a = \sup\{\ell, a\} = \ell \in J.
		\]
		Therefore $L \subseteq J$, and so $J_A \cup L \subseteq J$. Conversely, if $x \in J$ and $x \notin L$, then $x \in A$, and thus $x \in J_A$. Hence $J = J_A \cup L$. Note that $J_A$ is a contraction of $J$ and so it is an ideal of $A$. Moreover, if $J$ is prime in $S$, then its contraction $J_A = J \cap A$ is prime in $A$. Conversely, suppose $J_A$ is prime in $A$ and let $x, y \notin J$. Since $L \subseteq J$, we have $x, y \in A \setminus J_A$. Because $J_A$ is prime, $x \cdot_A y \notin J_A$, hence $x \cdot y \notin J$. Therefore $J$ is prime in $S$.
		
		\item If $J$ contains no nonzero element of $A$, then every nonzero element of $J$ lies in $L$. Since $L$ is well-ordered, $J \setminus \{0\}$ has a smallest element $\ell_0$. Then 
		\[
		J \subseteq \{0\} \cup \{\ell \in L : \ell \geq \ell_0\}.
		\]
		For any $\ell \geq \ell_0$, we have $\ell = \ell \cdot \ell_0 = \sup\{\ell, \ell_0\} \in J$, so the reverse inclusion holds. Thus 
		\[
		J = \{0\} \cup \{\ell \in L : \ell \geq \ell_0\}.
		\]
		We now show $J$ is prime. Take $x, y \in S$ with $x \cdot y \in J$. 
		\begin{itemize}
			\item If $x \cdot y = 0$, then $x = 0$ or $y = 0$ because $S$ is entire, so $x \in J$ or $y \in J$.
			\item If $x \cdot y = \ell \geq \ell_0$ (nonzero), then the multiplication lies in $L$. 
			\begin{itemize}
				\item If $x, y \in L$, then $x \cdot y = \sup(x,y) \in J$. Hence the larger of $x,y$ is $\geq \ell_0$, so it belongs to $J$.
				\item If one factor is in $A \setminus \{0\}$ and the other in $L$, say $x = a \in A \setminus \{0\}$, $y = \ell' \in L$, then $x \cdot y = \sup(a,\ell') = \ell' \in J$, so $y \in J$.
			\end{itemize}
		\end{itemize}
		Thus in all cases $x \in J$ or $y \in J$, so $J$ is prime.
	\end{enumerate}
	This completes the proof.
\end{proof}

\begin{corollary}
	Let $J$ be a nonzero ideal of the semiring $\mathscr{C}$ of cardinal numbers. Then exactly one of the following cases happens:
	
	\begin{itemize}
		\item $J$ contains a nonzero finite cardinal. Then 
		\[
		J = J_f \cup \{\kappa : \kappa \text{ is an infinite cardinal}\},
		\]
		where $J_f \trianglelefteq \mathbb{N}_0$ is the set of all finite cardinals in $J$. Furthermore, $J$ is a prime ideal of $\mathscr{C}$ if and only if $J_f$ is a prime ideal of $\mathbb{N}_0$.
		
		\item $J$ contains no nonzero finite cardinal. Then 
		\[
		J = \{0\} \cup \{\kappa \in \mathscr{C} : \kappa \geq \alpha\},
		\]
		where $\alpha$ is the smallest infinite cardinal in $J$. In this case, $J$ is prime.
	\end{itemize}
\end{corollary}

\begin{proof}
	The semiring $\mathbb{N}_0$ is a positive information algebra and the set of infinite cardinals is well-ordered by $\leq$ \cite[p. 48]{Jech2003}.
\end{proof}

\begin{remark}\label{fusionsemiringlatticecontractionnotinjective}
The contraction map $J \mapsto J \cap A$ in Proposition \ref{fusionsemiringlatticeextcontractionideals} is generally not injective on the full ideal lattice $\mathrm{Id}(S)$ of the overflow semiring $S$. Theorem \ref{fusionsemiringlatticeidealstructure} immediately provides counterexamples. For example, let $L$ be the two-element chain $\ell_1 < \ell_2$. By Theorem \ref{fusionsemiringlatticeidealstructure}(2), choosing $\ell_0=\ell_1$ and $\ell_0=\ell_2$ yields two distinct ideals of $S$ containing no nonzero elements of $A$:
	\[
	J_1=\{0\}\cup\{\ell\in L:\ell\ge \ell_1\}=\{0,\ell_1,\ell_2\},
	\]
	\[
	J_2=\{0\}\cup\{\ell\in L:\ell\ge \ell_2\}=\{0,\ell_2\}.
	\]
	Since
	\[
	J_1\cap A=\{0\}=J_2\cap A
	\]
	while $J_1\neq J_2$, the contraction map fails to be injective on $\mathrm{Id}(S)$.
\end{remark}

\begin{lemma}\label{maximalidealscontainL}
	Let $S = A \uplus_{\operatorname{ord}} L$ be an overflow semiring with $L \neq \emptyset$. Then every maximal ideal of $S$ contains $L$.
\end{lemma}

\begin{proof}
	If $\ell \in L \setminus \mathfrak{M}$, maximality gives $1 = m + \ell s$ with $m \in \mathfrak{M}$, $s \in S$. 
	For $s = 0$ we obtain $1 \in \mathfrak{M}$, impossible. 
	For $s \neq 0$, $\ell s \in L$, so $m + \ell s \in L$ because a sum with an $L$-term lies in $L$, yet $1 \in A$ with $A \cap L = \emptyset$, contradiction. 
	Thus $L \subseteq \mathfrak{M}$.
\end{proof}

\begin{theorem}\label{fusionsemiringlatticemaximalideal}
	Let $S = A \uplus_{\ord} L$ be the overflow semiring, with $L$ nonempty and well-ordered. If $\mathfrak{m}$ is a maximal ideal of $A$, then $\mathfrak{m} \cup L$ is a maximal ideal of $S$. Conversely, if $\mathfrak{M}$ is a maximal ideal of $S$, then $\mathfrak{M} \cap A$ is a maximal ideal of $A$. In particular, $A$ is local if and only if $S$ is local.
\end{theorem}

\begin{proof}
	Let $\mathfrak{m}$ be a maximal ideal of $A$. Then $\mathfrak{m} \cup L$ is an ideal of $S$ by Proposition \ref{fusionsemiringlatticeextcontractionideals}. Suppose $J$ is an ideal of $S$ properly containing $\mathfrak{m} \cup L$. Then $J$ contains some $a \notin \mathfrak{m} \cup L$. Since $A$ and $L$ are disjoint, $a \in A \setminus \mathfrak{m}$. Because $\mathfrak{m}$ is maximal in $A$, the ideal generated by $\mathfrak{m}$ and $a$ in $A$ is $A$ itself; hence $1_A \in A$ belongs to this ideal. Consequently $1_A \in J$, which forces $J = S$. Thus $\mathfrak{m} \cup L$ is maximal in $S$.
	
	Conversely, let $\mathfrak{M}$ be a maximal ideal of $S$. By Lemma \ref{maximalidealscontainL}, we have $L \subseteq \mathfrak{M}$. Hence $\mathfrak{M} = (\mathfrak{M} \cap A) \cup L$. Now $\mathfrak{M} \cap A$ is an ideal of $A$ by Proposition \ref{fusionsemiringlatticeextcontractionideals}. If $\mathfrak{M} \cap A$ were not maximal, there would exist a proper ideal $\mathfrak{n}$ of $A$ such that $\mathfrak{M} \cap A \subset \mathfrak{n} \subset A$. Then $\mathfrak{n} \cup L$ is an ideal of $S$ properly containing $\mathfrak{M}$ because
	\[
	\mathfrak{M} = (\mathfrak{M} \cap A) \cup L \subset \mathfrak{n} \cup L,
	\]
	and $\mathfrak{n} \cup L \neq S$ because $1_A \notin \mathfrak{n}$. This contradicts the maximality of $\mathfrak{M}$. Therefore $\mathfrak{M} \cap A$ is maximal in $A$.
	
	The final statement follows immediately: $A$ has a unique maximal ideal if and only if $S$ does, by the bijective correspondence established above.
\end{proof}

\begin{proposition}\label{fusionsemiringlatticelocal}
Let $A$ be a positive information algebra such that $a \geq 1$ for all nonzero $a \in A$. Let $L$ be a join-semilattice. Then the overflow semiring $S = A \uplus_{\ord} L$ is local and its unique maximal ideal is $\mathfrak{M} = S \setminus \{1\}$.
\end{proposition}

\begin{proof}
Let $\mathfrak{M} = S \setminus \{1\}$. Clearly $0 \in \mathfrak{M}$ and $1 \notin \mathfrak{M}$. If $x,y \in \mathfrak{M}$ and both lie in $A$ and one of them is $0$, their sum is the other; if both $>1$ then $x+y \ge x > 1$ (since $y \ge 0$). If at least one summand is in $L$, their sum is their supremum, which lies in $L$ and is $>1$. Hence $x+y \in \mathfrak{M}$. Now, let $s \in S$. If $s=1$ then clearly $sx \in \mathfrak{M}$. If $s \neq 1$, then $s \in \mathfrak{M}$. If both $s,x$ are in $A\setminus\{0\}$, then $sx \ge x \ge 1$. If $sx = 1$ then $x \le sx = 1$, forcing $x=1$, contradiction. Thus $sx \in \mathfrak{M}$. If one factor is in $L$ and the other is nonzero, their multiplication is the supremum, which lies in $L$ and is $>1$. If one factor is 0, then their multiplication is also 0. Thus in general $sx \neq 1$. Hence $\mathfrak{M}$ is an ideal. It is maximal because any strictly larger ideal must contain $1$ and thus equals $S$. Every proper ideal is contained in $\mathfrak{M}$, so $\mathfrak{M}$ is the unique maximal ideal. Hence $S$ is local.
\end{proof}

\begin{example}
	Observe that $B = \{0\} \cup [1, +\infty)$ is a subsemiring of $\mathbb{R}_{\geq 0}$ with the usual operations. Take any subsemiring $A$ of $B$, and let $L$ be the class of infinite cardinal numbers under the usual cardinal ordering. Then $S = A \uplus_{\ord} L$ satisfies the conditions of Proposition~\ref{fusionsemiringlatticelocal}.
\end{example}

Following the terminology for PM-rings (cf. \cite{Contessa1981}), we define PM-semirings as follows:

\begin{definition}\label{PMsemiringdef}
A semiring $S$ is called a PM-semiring if each prime ideal of $S$ is contained in a unique maximal ideal of $S$.
\end{definition}

\begin{proposition}\label{overflowsemiringPM}
Let $S = A \uplus_{\ord} L$ be the overflow semiring and let $L$ be well-ordered. Then the following statements are equivalent:
	\begin{enumerate}
		\item $A$ is local.
		\item $S$ is local.
		\item $S$ is a PM-semiring.
	\end{enumerate}
\end{proposition}

\begin{proof}
	$(1) \Leftrightarrow (2)$: Theorem \ref{fusionsemiringlatticemaximalideal}.
	
	$(2) \Rightarrow (3)$: A local semiring has a unique maximal ideal containing every prime ideal, hence is PM.
	
	$(3) \Rightarrow (1)$: Since $S$ is entire, the ideal $\{0\}$ is prime. In a PM-semiring, $\{0\}$ lies in a unique maximal ideal, so $S$ is local. By $(1)\Leftrightarrow(2)$, $A$ is local.
\end{proof}

\begin{proposition}\label{fusionsemiringlatticelocaldual}
	Let $A$ be a positive information algebra such that $a \leq 1$ for all $a \in A$, and assume that $a,b<1$ implies $a+b<1$. Let $L$ be a join-semilattice. Then the overflow semiring $S = A \uplus_{\ord} L$ is local with unique maximal ideal $\mathfrak{M} = S \setminus \{1\}$.
\end{proposition}

\begin{proof}
Set $\mathfrak{M} = S \setminus \{1\}$. Clearly $0\in\mathfrak{M}$ and $1\notin\mathfrak{M}$. If $x,y\in\mathfrak{M}$ and both lie in $A$, then $x,y<1$; by hypothesis $x+y<1$, so $x+y\in\mathfrak{M}$. If at least one lies in $L$, then $x+y=\sup\{x,y\}\in L$, hence $\neq 1$ and in $\mathfrak{M}$.
	
\medskip \noindent Now, let both $s \in S$ and $x\in\mathfrak{M}$ be nonzero. If $x\in L$, then $sx = \sup\{s,x\}\in L$ and so $sx \in \mathfrak{M}$. If $x\in A\setminus\{0\}$, then $x<1$. Now, if $s\in A$, monotonicity gives $sx\le 1\cdot x=x<1$, so $sx\neq1$. If $s\in L$, $sx=\sup\{s,x\}=s\in L$, hence $sx\neq1$. Thus $\mathfrak{M}$ is an ideal. Any proper ideal cannot contain $1$, therefore is contained in $\mathfrak{M}$. Hence $\mathfrak{M}$ is the unique maximal ideal and $S$ is local.
\end{proof}

\begin{example}
The restricted max-plus algebra $A = [-\infty, 0]$, with addition $x \oplus y = \max(x,y)$ and multiplication $x \odot y = x + y$, satisfies the conditions of Proposition~\ref{fusionsemiringlatticelocaldual}. Indeed, the additive identity is $0_A = -\infty$ and the multiplicative identity is $1_A = 0$. Every $a \in A$ with $a \neq 0_A$ (i.e., $a \in (-\infty, 0]$) satisfies $a \leq 1_A$. Moreover, if $a,b < 1_A$, then $a,b < 0$, and their sum $a \oplus b = \max(a,b) < 0 = 1_A$.
\end{example}

\begin{definition}\label{Krulldimensiondef}
Let $S$ be a semiring. The Krull dimension of $S$, denoted $\dim S$, is the supremum of all integers $n \geq 0$ such that there exists a strictly ascending chain
	\[
	P_0 \subset P_1 \subset \cdots \subset P_n
	\]
	of prime ideals of $S$. If no such finite supremum exists, we set $\dim S = \infty$.
\end{definition}

\begin{theorem}\label{Krulldimensionoffusion}
	Let $S = A \uplus_{\ord} L$ be the overflow semiring, where $L$ is a well-ordered set. Then the Krull dimension of $S$ is given by
	\[
	\dim S = 
	\begin{cases}
		\dim A, & \text{if } L = \varnothing \text{ and } \dim A < \infty,\\[4pt]
		\infty, & \text{if } L \text{ is infinite} \text{ or } \dim A = \infty,\\[4pt]
		|L| + \dim A, & \text{if } L \text{ is finite and nonempty and } \dim A < \infty.
	\end{cases}
	\]
\end{theorem}

\begin{proof}
	If $L = \varnothing$, then $S = A$, so $\dim S = \dim A$. If $\dim A = \infty$, then $S$ contains an infinite chain from $A$ lifted via $I \mapsto I \cup L$, so $\dim S = \infty$. If $L$ is infinite, then for any $m \in \mathbb{N}$ choose $\ell_1 < \cdots < \ell_m$ in $L$ and a chain $P_0 \subset \cdots \subset P_n$ in $A$; then
	\[
	I_{\ell_m} \subset \cdots \subset I_{\ell_1} \subset P_0 \cup L \subset \cdots \subset P_n \cup L
	\]
	is a chain of length $m + n$. Since $m$ is arbitrary, $\dim S = \infty$.
	
	Now suppose $L$ is finite and nonempty with $|L| = k$ and $\dim A = n < \infty$. Since $\dim A = n$, there exists a chain $P_0 \subset \cdots \subset P_n$ of primes in $A$. Lifting gives $P_0 \cup L \subset \cdots \subset P_n \cup L$ in $S$. List $L$ increasingly: $\ell_1 < \cdots < \ell_k$. For each $\ell_i$, the ideal $I_{\ell_i} = \{0\} \cup \{\ell \in L : \ell \geq \ell_i\}$ is prime (Type II), and $\ell_i < \ell_{i+1}$ implies $I_{\ell_{i+1}} \subset I_{\ell_i}$. Moreover, $I_{\ell_1} \subset P_0 \cup L$. Concatenating,
	\[
	I_{\ell_k} \subset \cdots \subset I_{\ell_1} \subset P_0 \cup L \subset \cdots \subset P_n \cup L
	\]
	is a chain of length $k + n$, so $\dim S \geq k + n$. By Theorem \ref{fusionsemiringlatticeidealstructure}, every prime ideal of $S$ is either Type I ($P \cup L$) or Type II ($I_\ell$), and any chain can contain at most $k$ Type II primes (all of them, in reverse order) followed by at most $n+1$ Type I primes. Therefore $\dim S \leq k + n$. Hence, $\dim S = k + n$, as claimed.
\end{proof}

\begin{examples} We illustrate the Krull dimension of the overflow semiring $S = A \uplus_{\ord} L$ below:
	
	\begin{enumerate}
		\item Let $A$ be a zerosumfree positive semifield (so $\dim A = 0$) and let $L$ be a finite chain of length $m$ (i.e., $|L| = m$). Then $\dim S = m$.
		
		\item The Krull dimension of the semiring $\mathscr{C}$ of cardinal numbers is $\infty$ because we have infinitely many cardinal numbers, including $\beth_0 = \aleph_0$ and $\beth_k = 2^{\beth_{k-1}}$ for every positive integer $k$ \cite[p. 55]{Jech2003}.
	\end{enumerate}
	
\end{examples}

\begin{lemma}\label{fusionsemiringlatticeidealgeneratingset}
	Let $S = A \uplus_{\ord} L$ be an overflow semiring with $L$ well-ordered and $I \trianglelefteq S$.
	\begin{itemize}
		\item If $I \cap A = \{0\}$, then $I = (\ell_0)$ with $\ell_0 = \min(I \cap L)$.
		\item If $I$ contains a nonzero $a \in A$, then $I = I_A \cup L$ where $I_A = I \cap A$. Moreover:
		\begin{itemize}
			\item If $G$ generates $I_A$ in $A$, then $G \cup \{\min L\}$ generates $I$ in $S$.
			\item If $H$ generates $I$ in $S$, then $H \cap A$ generates $I_A$ in $A$.
		\end{itemize}
	\end{itemize}
\end{lemma}

\begin{proof}
	The first case is Theorem \ref{fusionsemiringlatticeidealstructure}. For the second, $I = I_A \cup L$ by Theorem \ref{fusionsemiringlatticeidealstructure}. If $G$ generates $I_A$ in $A$, then $G \cup \{\min L\}$ generates $I$ because every $\ell \in L$ equals $\ell \cdot \min L$. Conversely, if $H$ generates $I$ in $S$, any $a \in I_A$ is a finite sum $\sum s_i h_i$. Terms landing in $L$ would force $a \in L$, impossible; hence each $s_i, h_i \in A$, so $H \cap A$ generates $I_A$ in $A$.
\end{proof}

Recall that a semiring $S$ is Noetherian if and only if every ideal of $S$ is finitely generated \cite[Proposition 6.16]{Golan1999}.

\begin{theorem}\label{fusionsemiringlatticeNoetherian}
	Let $S = A \uplus_{\ord} L$ be the overflow semiring, where $A$ is a positive information algebra and $L$ is a well-ordered set. Then $S$ is Noetherian if and only if $A$ is Noetherian.
\end{theorem}

\begin{proof}
$(\Rightarrow)$ Suppose $S$ is Noetherian. Assume, toward a contradiction, that $A$ is not Noetherian. Then there exists an ideal $I_A$ of $A$ that is not finitely generated. Since the zero ideal $\{0\}$ is finitely generated, $I_A$ must contain a nonzero element of $A$. By Theorem \ref{fusionsemiringlatticeidealstructure}, $I=I_A\cup L$ is an ideal of $S$. By Lemma \ref{fusionsemiringlatticeidealgeneratingset}, $I$ is finitely generated in $S$ if and only if $I_A$ is finitely generated in $A$. This contradicts the assumption that $S$ is Noetherian. Hence $A$ is Noetherian.
	
$(\Leftarrow)$ Conversely, assume $A$ is Noetherian, and let $I$ be an ideal of $S$. If $I$ contains no nonzero element of $A$, then by Theorem \ref{fusionsemiringlatticeidealstructure}, $I=(\ell_0)$, where $\ell_0=\min(I\cap L)$. Thus $I$ is principal. Now suppose $I$ contains a nonzero element of $A$. Then $I=I_A\cup L$, where $I_A=I\cap A$ is an ideal of $A$. Since $A$ is Noetherian,
	\[
	I_A=(a_1,\dots,a_n)
	\]
	for some $a_1,\dots,a_n\in A$. By Lemma \ref{fusionsemiringlatticeidealgeneratingset},
	\[
	I=(a_1,\dots,a_n,\ell_{\min}),
	\]
	where $\ell_{\min}=\min L$. Hence $I$ is finitely generated. Therefore every ideal of $S$ is finitely generated, so $S$ is Noetherian.
\end{proof}

\begin{corollary}\label{CNoetherian}
Let $L$ be a well-ordered set. Then the overflow semiring $S= \mathbb{N}_0 \uplus_{\ord} L$ is a Noetherian semiring. In particular, the class-semiring of cardinal numbers is Noetherian.
\end{corollary}

\begin{proof}
It is easy to see that $I$ is an ideal of the semiring $\mathbb{N}_0$ if and only if $I$ is a submonoid of $(\mathbb{N}_0,+)$. By Corollary 2.8 in \cite{RosalesSanchez2009}, every submonoid of $\mathbb{N}_0$ is finitely generated. It follows that the semiring $\mathbb{N}_0$ is Noetherian. Thus by Theorem \ref{fusionsemiringlatticeNoetherian}, $S$ is Noetherian finishing the proof. 
\end{proof}

Recall that a semiring $S$ is called Artinian if $S$ satisfies the descending chain condition on its ideals (see Definition 1.1 in \cite{MukherjeeSenGhosh1996}).

\begin{theorem}\label{fusionsemiringlatticeArtinian}
Let $S = A \uplus_{\ord} L$ be the overflow semiring, where $A$ is a positive information algebra and $L$ is a well-ordered set. Then $S$ is Artinian if and only if:
	\begin{enumerate}
		\item $A$ is Artinian;
		\item $L$ is finite.
	\end{enumerate}
\end{theorem}

\begin{proof}
	$(\Rightarrow)$ If $S$ is Artinian, any descending chain in $A$ lifts to one in $S$ via $I\mapsto I\cup L$, so $A$ is Artinian. If $L$ were infinite, well-orderedness yields $\ell_1<\ell_2<\cdots$; then $I_{\ell_k}=\{0\}\cup\{\ell\in L:\ell\ge\ell_k\}$ gives $I_{\ell_1}\supset I_{\ell_2}\supset\cdots$, a contradiction. Hence $L$ is finite.
	
	$(\Leftarrow)$ Assume $A$ Artinian and $L$ finite. Let $J_1\supseteq J_2\supseteq\cdots$ be ideals in $S$. Once some $J_k$ contains no nonzero $A$-element (Type II), all later ones do as well. Thus the chain splits: finitely many Type I ideals $J_k=I_k\cup L$ ($I_k\trianglelefteq A$), then Type II ideals. The $I_k$ stabilize because $A$ is Artinian. For Type II ideals $J_k=\{0\}\cup\{\ell\in L:\ell\ge\ell_k\}$, we have $\ell_k\le\ell_{k+1}$; as $L$ is finite, the $\ell_k$ stabilize. Hence the original chain stabilizes, so $S$ is Artinian.
\end{proof}

\begin{example}\label{CnotArtinian}
Note that in the semiring $\mathbb{N}_0$, the ideals generated by $2^n$ form an infinite strictly descending chain:
\[
(1) \;\supset\; (2) \;\supset\; (4) \;\supset\; (8) \;\supset\; \cdots
\]
Hence $\mathbb{N}_0$ is not Artinian. Therefore by Theorem \ref{fusionsemiringlatticeArtinian}, for any well-ordered set $L$, the overflow semiring $S = \mathbb{N}_0 \uplus_{\ord} L$ is not Artinian.
\end{example}

\begin{remark}
The overflow semiring $S = A \uplus_{\ord} L$ is structurally analogous to the indigenous semirings
\[
S_k = \{0, 1, \dots, k, m\}
\]
introduced in \cite{BehzadipourKoppelaarNasehpour2025}. Both frameworks model computational overflow: The indigenous semiring $S_k$ truncates arithmetic above $k$ into a monolithic element $m$ and the overflow semiring $S$ expands this overflow zone into an arbitrary join-semilattice $L$ dominating a general positive information algebra $A$. Algebraically, both structures are entire and zerosumfree, their upper-layer elements collapse into idempotents, and both are austere. However, whereas $S_k$ is a finite local semiring with a fixed Krull dimension ($\dim S_k = 1$), $S$ possesses a far richer prime ideal spectrum, dynamic Noetherian and Artinian properties, and a customizable Krull dimension governed by the formula
\[
\dim S = \dim A + |L|
\]
for finite chains $L$.
\end{remark}

\section*{Acknowledgments}

The author would like to thank Prof. Dr. Henk Koppelaar for reviewing the manuscript and providing helpful comments and suggestions, which contributed to improving the presentation of this paper.

	\bibliographystyle{plain}

\end{document}